\documentclass[10pt, leqno]{amsart}

\usepackage[utf8]{inputenc}
\usepackage{amsfonts}
\usepackage{stackrel}
\usepackage{hyperref}
\usepackage{wrapfig}
\usepackage{listings} 
\usepackage{amsmath,amsthm,amsfonts,amssymb,fancyhdr,graphics,enumerate,color,xfrac,tikz,yfonts,float,mathabx}

\usepackage{mathtools}
\usepackage{pgf,tikz}
\usepackage{tikz-cd}
\usetikzlibrary{arrows}

\theoremstyle{plain}

\newtheorem{lemma}{Lemma}[section]
\newtheorem{proposition}[lemma]{Proposition}
\newtheorem{theorem}[lemma]{Theorem}
\newtheorem{corollary}[lemma]{Corollary}

\newtheorem*{subclaim*}{Subclaim}

\theoremstyle{remark}

\newtheorem{remark}[lemma]{Remark}
\newtheorem*{remark*}{Remark}

\newtheorem*{notation*}{Notation}
\newtheorem{question}[lemma]{Question}
\newtheorem*{question*}{Question}

\theoremstyle{definition}

\newtheorem*{definition*}{Definition}
\newtheorem*{remarks*}{Remarks}

\let\strokeL\L

\newcommand*{\tcap}{\mathbin{\scalebox{1.5}{\ensuremath{\cap}}}}%

\def\N{\mathbb{N}}
\def\Z{\mathbb{Z}}
\def\Q{\mathbb{Q}}

\def\F{\mathbb{F}}

\def\pp{\mathbf{p}}
\def\t{\mathbf{t}}

\def\L{\mathcal{L}}
\def\U{\mathcal{U}}
\def\E{\mathcal{E}}

\def\FL{\textbf{F}}
\def\X{\textbf{X}}
\def\Y{\textbf{Y}}
\def\0{\mathbf{0}}

\def\a{\mathfrak{a}}
\def\m{\mathfrak{m}}
\def\n{\mathfrak{n}}
\def\q{\mathfrak{q}}
\def\p{\mathfrak{p}}

\def\UG{\mathfrak{UGroup}}
\def\pLie{\mathfrak{pLie}}
\def\La{\mathfrak{L}}

\DeclareMathOperator{\charac}{char}

\DeclareMathOperator{\GL}{GL}

\DeclareMathOperator{\Mat}{M}
\DeclareMathOperator{\kdim}{dim_{Krull}}
\DeclareMathOperator{\Kdim}{dim_{Krull}}

\title{Linearity of compact $R$-analytic groups}

\author[M. Casals-Ruiz]{Montserrat Casals-Ruiz}
\address{Ikerbasque - Basque Foundation for Science and Department of Mathematics, University of the Basque Country (UPV/EHU), Bilbo 48080, Spain}
\email{montserrat.casals@ehu.eus}

\author[A. Zozaya]{Andoni Zozaya}
\address{Department of Mathematics, University of the Basque Country (UPV/EHU), Bilbo 48080, Spain}
\email{andoni.zozaya@ehu.eus}

\date{}

\begin{document}

\begin{abstract}
We prove that any compact $R$-analytic group is linear when $R$ is a pro-$p$ domain of characteristic zero. 
\end{abstract}

\thanks{This work has been supported by the Spanish Government, grants PID2020-117281GB-I00 and  PID2019-107444GA-I00, partly by the European Regional Development Fund (ERDF), and the Basque Government, grant IT1483-22. The second author is supported as well by Spanish Ministry of Science, Innovation and Universities' grant FPU17/04822.}


\maketitle

\section{Introduction}

The study of $p$-adic analytic pro-$p$ groups has given rise to a prolific mathematical subject since Lazard \cite{La68} introduced them in 1968, see \cite{DDMS} and the references therein. At the core of its development lies their rich algebraic structure and the myriad of properties they satisfy: they have polynomial subgroup growth, have finite rank, they are isomorphic to a closed subgroup of a Sylow pro-$p$ subgroup of $\GL_n(\Z_p),$ etc.

The theory of analytic profinite groups developed further by considering manifolds over more general rings. For instance, in \cite[Chapter III]{Bou} and \cite[Part II]{Serre} analytic groups over local principal ideal domains are studied. More generally, let $R$ be a pro-$p$ domain, namely a local Noetherian integral domain, with maximal ideal say $
\m,$ that is complete with respect to the $\m$-adic topology and whose residue field $R/\m$ is finite of characteristic $p$ (e.g. the power series rings $\Z_p[[t_1,\dots, t_m]]$ and $\F_p[[t_1, \dots, t_m]],$ where $\Z_p$ and $\F_p$ stand for the ring of $p$-adic integers and the finite field of $p$ elements respectively). In \cite[Chapter 13]{DDMS}, the authors introduce and delve into the study of $R$-analytic groups for a pro-$p$ domain $R.$ These comprise a topological group equipped with an analytic manifold structure over $R$ such that both structures are compatible in the sense that the multiplication and the inversion are analytic maps (see \cite[Definition 13.3]{DDMS}). This notion naturally generalizes the concept of $p$-adic analytic group, where $R$ is taken to be $\Z_p.$

Interest in more general families analytic groups arose when it was realised that they play a r\^{o}le in combinatorial group theory or problems concerning group words. For example, Lubotzky and Shalev \cite{LS94} studied $R$-perfect groups --a special class of $R$-analytic groups-- proving that their subgroup growth is close to polynomial and that they satisfy Golod-Shafarevich's inequality. Subsequently, Jaikin-Zapirain and Klopsch \cite{JaKl} continued the investigation of the structure of analytic groups, mainly for rings of positive characteristic, and among others, proved that finitely generated $\F_p[[t]]$-analytic groups are verbally elliptic. Furthermore, Bradford \cite{Bra18} employed analytic groups to obtain polylogarithmic upper bounds for the diameters of some finite simple groups of Lie type.  

The linearity of compact $p$-adic analytic groups is well-established \cite{La68}, but it is a long-standing question whether compact $R$-analytic groups are linear.  This question was posed two decades ago by Lubotzky and Shalev (see \cite[Question 2, p. 311]{LS94}), and the same remark has been made in \cite[Question 5, p. 124]{DuSa}.

There are several partial results addressing the aforementioned problem, and the question has positive answer for just-infinite $R$-analytic groups (Jaikin-Zapirain, \cite{Ja02}) and  perfect $\Z_p[[t]]$-standard groups (Camina and Du Sautoy, \cite{CD}); refer to Subsection \ref{standard} for the definition of standard group. In a broader scope, \cite{Ja06} proves that when $R$ is a pro-$p$ domain of characteristic zero, finitely generated $R$-analytic groups are linear.

It should be noted that perfect $R$-analytic groups are finitely generated, and, consequently, the majority of the known cases fall into the class of finitely generated profinite groups. Nonetheless, within the setting of $R$-analytic groups, this condition is quite constraining, in view that every (non-discrete) finitely generated compact $R$-analytic group that satisfies a group identity is in fact $p$-adic analytic, see \cite[Theorem 1.3]{JaKl}.

In this paper, we give a full answer to this question for pro-$p$ domains of characteristic zero. More precisely, we prove the following:

\begin{theorem}\label{main}
Let $R$ be a pro-$p$ domain of characteristic zero. Then every compact $R$-analytic group is linear.
\end{theorem}

 Our proof has a model-theoretic flavour. The main observation is that any compact $R$-standard group $G,$ where $R$ is a pro-$p$ domain of characteristic $0$, is discriminated by $\GL_n(\Z_p)$ for a suitable $n,$ see Corollary \ref{cor:disc}. This result relies on the linearity of uniform pro-$p$ groups, see Section \ref{subsec: uniform}. Recall that we say that $G$ is discriminated by $H$ (equivalently, $G$ is fully residually $H$) if, for any finite set $S$ of elements of $G$, there exists a group homomorphism $h_S: G \to H$ injective on $S$. Accordingly, from standard model-theoretic results, we deduce that $G$ is a subgroup of the ultrapower of $\GL_n(\Z_p)$ which is isomorphic to $\GL_n(\Z_p^\U)$ where $\Z_p^\U$ is the ultrapower of $\Z_p$ and, so in particular, it is linear, see Section \ref{subsec: model theory}.


\section{Proof of Theorem \ref{main}}

\subsection{Notation and conventions}

For a ring $Q$, we denote by $\charac{Q}$ and $\Kdim{Q}$ the characteristic and the Krull dimension of $Q$. Moreover, $Q[[X_1, \dots, X_n]],$ or simply $Q[[\X]],$ is the ring of formal power series with coefficients in $Q,$ that is, power series $\sum_{\alpha} a_\alpha \X^{\alpha},$ where $\alpha  = (\alpha_1, \dots, \alpha_n) \in \N_0^n,$ $a_{\alpha} \in Q$ and $\X^{\alpha}$ is short for $X_1^{\alpha_1}\dots X_n^{\alpha_n}.$

Throughout the paper $p$ will be a prime number and $R$ will be a pro-$p$ domain with maximal ideal $\m.$

For an ideal $\a$ both the $N$-th power ideal and the $N$-Cartesian power of $\a$ are used regularly in the paper. It is customary to use the notation $\a^N$ to refer to both concepts. In order to distinguish them, we keep the notation $\a^N$ for the $N$-Cartesian power of the ideal $\a,$ and we fix the notation $\a^{*N}$ to denote the $N$-th power ideal of $\a$.


\subsection{Reduction to \texorpdfstring{$R$}{R}-standard groups} \label{standard} 

Any compact $R$-analytic group contains a so-called \emph{$R$-standard} group of finite index (see \cite[Theorem 13.20]{DDMS}), which has a strong algebraic structure. 
By virtue of the induced linear representation, the proof of the principal theorem is going to be reduced to this subgroup, so we next recall the definition. 

An $R$-analytic group $S$ of dimension $d$ is $R$-standard if there exist
\begin{itemize}
\item[(i)]  a homeomorphism $\phi \colon S \rightarrow \left( \m^{*N}\right)^{d},$ for some $N \in \N,$ such that $\phi(1) = \0$ and 
\item[(ii)] some formal power series $F_{j} \in R[[X_1, \dots, X_{2d}]],$ $j \in \{1, \dots, d\},$   such that 
$$\phi(x y) = \left(  F_1(\phi(x), \phi(y)), \dots, F_d(\phi(x), \phi(y))\right) \ \forall x, y \in S.$$
\end{itemize}
The integer $N$ in (i) is the \emph{level} of $S,$ and the tuple of power series $\FL =(F_1, \dots, F_d)$ in (ii) is called the \emph{formal group law} of $S$ with respect to $\phi$. Any $R$-standard group is a pro-$p$ group, and so compact (see \cite[Proposition 13.22]{DDMS}).


\subsection{Matrix representation of uniform pro-\texorpdfstring{$p$}{p} groups}

\label{subsec: uniform}

For compact $p$-adic analytic groups, $\Z_p$-standard groups are precisely uniform pro-$p$ groups, namely finitely generated torsion-free pro-$p$ groups $G$ such that $[G, G] \leq G^{\pp},$ where $\pp = p$ when $p$ is odd and $\pp = 4$ when $p =2$ ---we shall keep this definition of $\pp$ for the rest of the section---.

The linearity of uniform pro-$p$ groups is an essential stepping stone in our work. The proof appeals to Ado's Theorem in conjunction with the Baker-Campbell-Hausdorff formula in order to link a $p$-adic analytic group with its Lie algebra. More precisely:

\begin{theorem}[cf. \textup{\cite[Section 7.3]{DDMS}}]
\label{dim1linear}
Let $G$ be a uniform pro-$p$ group of dimension $d.$ There exists a group monomorphism $ m \colon G \hookrightarrow \GL_{n}(\Z_p)$ where $n \leq \gamma(p, d)$ for a function $\gamma \colon \N^2 \rightarrow \N.$   
\end{theorem}

\begin{proof}
There exists a categorical isomorphism $\L \colon \UG \rightarrow \pLie$ between the categories $\UG$ of $d$-dimensional uniform pro-$p$ groups and the category $\pLie$ of $d$-dimensional powerful $\Z_p$-lattices, namely $\Z_p$-Lie algebras $\La$ that are free $\Z_p$-modules of rank $d$ and satisfy  $[\La, \La] \leq \pp \La$  (see \cite[Theorem 9.10]{DDMS}). Let $\mathcal{E}$ be the inverse of $\L.$ Then, $\E$ is defined by the Baker-Campbell-Hausdorff formula, and for the powerful matrix $\Z_p$-lattice  $\pp \Mat_\ell(\Z_p)$, $\E$ is nothing but the usual matrix exponentiation. In particular, 
$$\E(\pp \Mat_\ell(\Z_p)) = I_\ell  + \pp \Mat_\ell(\Z_p) \leq \GL_\ell(\Z_p)$$
(compare with \cite[Chapter II, \S \ 8, Proposition 4]{Bou}).

 Since $\L(G)$ is a $\Z_p$-lattice of dimension $d,$ by virtue of Ado's Theorem for principal ideal domains (see \cite[Theorem 1.1]{A4}), there exists a faithful Lie algebra representation $\phi \colon \L(G) \rightarrow \Mat_{\ell}(\Z_p)$ where $\ell \leq f(d)$ for some function $f \colon \N \rightarrow \N.$ 

On the one hand, $\mathcal{E}$ is a categorical isomorphism so 
$$\E\left(\phi \right) \colon \E\left(\pp\L(G) \right) \hookrightarrow \E(\pp \Mat_\ell(\Z_p)) \leq \GL_\ell(\Z_p)$$
is a group monomorphim. On the other hand, by \cite[Proposition 4.31(iii)]{DDMS}, the additive cosets --as algebras-- coincide with the multiplicative cosets --as groups--, and so in particular,
$$\left| G : \E\left(\pp\L(G)\right) \right| = \left| \L(G) : \pp \L(G) \right| = \left|\Z_p^d : \pp \Z_p^d \right| = \pp^d.$$
Therefore, considering the induced representation, there exists a group monomorphism 
$$m: G \hookrightarrow \GL_{n}(\Z_p),$$
where $n \leq  \gamma(p, d)$ for $\gamma(p, d) = \pp^d f(d).$
\end{proof}

Although the proof is essentially that of Lazard \cite{La68}, the principal improvement comparing with \cite[Theorem 7.19]{DDMS} and \cite[Theorem C]{Wei} is the bound on the degree of the linear representation in terms of the dimension $d$ and the prime $p$.

The difference with the case of positive characteristic arises in this subsection, as a categorical isomorphism between standard groups and a suitable category of Lie lattices might not exist over pro-$p$ domains of positive characteristic.

\begin{question}
Let $G$ be an $\F_p[[t]]$-standard group of dimension $d$. Do they exist a field $K$ and a group monomorphism $m \colon G \hookrightarrow \GL_n(K),$ where $n$ is bounded only in terms of $d$ and $p$? 
\end{question}

In view of the techniques outlined in this paper, a positive answer to the question above would imply the linearity of compact $R$-analytic groups over a general pro-$p$ domain $R$ of positive characteristic.


\subsection{Discrimination}
\label{discrimination}

If $A$ and $B$ are two instances of the same algebraic structure, then $A$ is said to be \emph{fully residually} $B$ or $A$ is \emph{discriminated} by $B$ --equivalently, $B$ discriminates $A$-- if for any finite subset $S \subseteq A,$ there exists a homomorphism $h \colon A \rightarrow B$ in the corresponding category such that the restriction $h|_S$ on the set $S$ is injective.

For instance, if $(R, \m)$ is a pro-$p$ domain and $\X$ an $m$-tuple of variables, then $R[[\X]]$ is discriminated by $R.$ Indeed, for any finite subset $S \subseteq R[[\X]]$ there exists $a \in \m^{m}$ such that the continuous \emph{evaluation epimorphism} $s_a \colon R[[\X]] \rightarrow R,$ $F(\X) \mapsto F(a)$ is injective in $S$ (see \cite[Lemma 9]{Ja06}).  Drawing from this idea we have:

\begin{lemma}
\label{ring discrimination}
Let $R$ be a pro-$p$ domain. For each finite subset $S \subseteq R,$ there exists a pro-$p$ domain $Q_S$ of Krull dimension one and a continuous ring epimorphism $\pi \colon R \rightarrow Q_S$ such that $\pi|_S$ is injective.
\end{lemma}
\begin{proof}
Let $S= \{ r_i\}_{i \in I} \subseteq R$ be a finite set of distinct elements, and define $r = \prod_{i \neq j} (r_i - r_j) \in R.$ Notice that $r\ne 0$ since the $r_i$ are distinct and $R$ is a domain.
  
According to Cohen's Structure Theorem \cite{Cohen}, $R$ is a finite integral extension of $P[[t_1, \dots, t_m]]$ where $m = \Kdim{R} -1,$ and $P$ is either $\Z_p$ or $\F_p[[t]]$ depending on the characteristic of $R$. 

For each $a \in \m^m,$ let $s_a \colon P[[\t]] \rightarrow P,$ $F(\t) \mapsto F(a)$ be the evaluation epimorphism and $\p_a = \ker{s_a}.$ By virtue of the  Going Up Theorem (see \cite[Theorem V.2.3]{ZaSa}), there exists a prime ideal $\q_a \subseteq R$ such that $\q_a \cap P[[\t]] = \p_a.$ Then $Q_a = \sfrac{R}{\q_a}$ is a pro-$p$ domain which is a finite integral extension of $P$ (see \cite[Remark to Lemma 4.3]{A2}), and the quotient map $\tilde{s}_a \colon R \rightarrow \sfrac{R}{\q_a} = Q_a$ extends $s_a.$ That is, the following diagram commutes with respect to the identification  $\iota: P \cong \sfrac{P[[t]]}{\p_a} \hookrightarrow Q_a = \sfrac{R}{\q_a},$ $x + \p_a \mapsto x + \q_a:$

$$\begin{tikzcd}
R \arrow[r, "\tilde{s}_a"]                 & Q_a                      \\
P[[t_1, \dots, t_m]] \arrow[r, "s_a"] \arrow[u, hook] & P \arrow[u, hook, "\iota"] 
\end{tikzcd}$$

Since $Q_a$ is an integral extension of $P,$ then $\Kdim{Q_a} =1.$ 

By \cite[Lemma 9]{Ja06}, we have that $\tcap_{a \in \m^m} \ker{s_a} = \{0\},$ and since $R$ is an integral extension of $P[[\t]],$ it follows from \cite[Complement 1 to Theorem V.2.3]{ZaSa} that $\tcap_{a \in \m^m} \ker{ \tilde{s}_a} = \{0\}.$ Hence, there exists $a \in \m^m$ such that $\tilde{s}_a(r) \neq 0,$ and thus $\tilde{s}_a(r_i) \neq \tilde{s}_a(r_j)$ for all $i \neq j \in I.$ 
\end{proof}

\begin{remark}
\label{remark}
As mentioned in the proof of the preceding theorem, if $R$ has characteristic zero, it is a finite integral extension of $\Z_p[[t_1, \dots, t_m]],$ where $m +1 = \kdim{R}.$  Let $\mu(R)$ be the minimum number of elements that generate $R$ as a $\Z_p[[t_1, \dots, t_m]]$-module. By construction $Q_a$ is generated as $\Z_p$-module by the image of the generators of $R$ as $\Z_p[[t_1, \dots, t_m]]$-module. Hence, as $\Z_p$ is a principal ideal domain, $Q_a$ is a free $\Z_p$-module of rank at most $\mu(R).$  
\end{remark}

\begin{proposition}
\label{discriminated}
Let $R$ be a pro-$p$ domain of characteristic zero, and let $G$ be an $R$-standard group. There exists an integer $n \in \N,$ depending on $R,$ the dimension of $G$ and the level of $G,$ such that $G$ is discriminated by $\GL_{n}(\Z_p).$
\end{proposition}
\begin{proof}
Since $G$ is an $R$-standard group of dimension $d$ and level $N,$ it can be identified with $\left(  \m^{*N}\right)^{d}$ such that the group operation is given by a formal group law with components
$$F_i(\X,\Y) = \sum_{\alpha, \beta \in  \N_0^d} a_{\alpha, \beta}^i \X^\alpha \Y^{\beta} \in R[[\X, \Y]]$$
(here $\X$ and $\Y$ denote $d$-tuples of variables).
Let $S \subseteq \left(  \m^{*N}\right)^{d}$ be a finite set of $d$-tuples, i.e.
$$S = \left\{ (r_{i1}, \dots, r_{id}) \in R^{d}\right\}_{i \in I},$$
and let $S' = \{ r_{ij}\}_{\substack{i \in I \\ j = 1, \dots, d}} \subseteq R.$  By Lemma \ref{ring discrimination}, there exists a pro-$p$ domain $Q$ of Krull dimension one with maximal ideal $\n$, and a continuous ring epimorphism $\pi \colon R \rightarrow Q$ that is injective when restricted to $S'.$ 

Let $H = \left( \n^{*N} \right)^{d}$ with the natural $Q$-analytic manifold structure. Then it can be endowed with a group structure, where the group operation is given by the $d$-dimensional formal group law with components
$$\tilde{F}_i(\X, \Y) =   \sum_{\alpha, \beta \in \N_0^d} \pi(a_{\alpha, \beta}^i) \X^\alpha \Y^{\beta} \in Q[[\X, \Y]]$$
(see \cite[Corollary 3.2]{A2}). Then $\pi^{d} \colon G \rightarrow H,$ $(r_1, \dots, r_d) \mapsto (\pi(r_1), \dots, \pi(r_d))$ is a group epimorphism that is injective when restricted to $S.$ 

Moreover, by Remark \ref{remark}, $Q$ is a free $\Z_p$-module of rank $\mu' \leq \mu(R)$, and thus, by restriction of scalars (see \cite[Examples 13.6(iv)]{DDMS}), $H$ is a $p$-adic analytic group of dimension $\mu'd.$ More precisely, if $\sigma \colon Q \rightarrow \Z_p^{\mu'}$ is a $\Z_p$-module isomorphism, then $p^{N} \Z_p^{\mu'} \subseteq \sigma\left(\m^{*N}\right).$ Thus, $H$ contains a uniform pro-$p$ group  of level $N,$ say
$$U := \left(\sigma^d \right)^{-1}\left( p^N \Z_p^{\mu' d} \right).$$

According to Theorem \ref{dim1linear}, there exists a faithful linear representation $m_1 \colon U \hookrightarrow \GL_{\gamma(p, \ \mu(R)d)}(\Z_p),$ for a function $\gamma \colon \N^2 \rightarrow \N.$ 

Furthermore, in view of \cite[Proposition 4.31 (iii)]{DDMS}, the subgroup indexes coincide with the indexes as additive groups, so
$$|H : U| \leq \left|\Z_p^{\mu' d} :  p^{N} \Z_p^{\mu' d} \right| = p^{N \mu' d} \leq p^{N \mu(R) d}.$$ 

 Let 
$$n = p^{N \mu(R) d}\gamma(p, \mu(R)d) \in \N $$ and let $m\colon H \hookrightarrow  \GL_{n}(\Z_p)$ be the linear representation induced from $m_1$. In particular, $m$ is a group monomorphism.
Lastly, $m \circ \pi^{d} \colon G \rightarrow \GL_{n}(\Z_p)$ is a group homomorphism that is injective when restricted to $S.$ 
\end{proof}

In other words, we have the following

\begin{corollary}\label{cor:disc}
Let $R$ be a pro-$p$ domain of characteristic $0$. Every $R$-standard group is discriminated by a linear group.
\end{corollary}


\subsection{Embedding into ultrapowers}
\label{subsec: model theory}

In order to conclude the proof we use some basics of Model Theory. We refer the reader to \cite{CK} for definitions and background. The main result is a corollary of the following theorem.

\begin{theorem}
\label{embeding}
Let $R$ be a pro-p domain of characteristic $0$. Let $G$ be an $R$-standard group, then $G$ embeds in an ultrapower of $\GL_n(\Z_p)$ for a suitable $n \in \N.$
\end{theorem}

Throughout the rest of the paper we write $A^\U$ for the ultrapower of an algebraic structure $A$ with respect to a non-principal ultrafilter $\U$ in $I,$ that is, $A^\U = \sfrac{\prod_{i \in I} A}{\U}.$
\begin{remark}
\label{existentially closed}
In the first-order language of groups, an \emph{existential sentence} is a sentence that has a prenex normal form of the type
$$\exists x_1, \dots, x_n \ \bigvee_{k\in K} \left(\bigwedge_{i \in I} u_i(x_1, \dots, x_n) = 1 \ \bigwedge_{j \in J} v_j(x_1, \dots, x_n) \neq 1\right),$$
where $u_i$ and $v_j$ are group words and the sets $I,J$ and $K$ are finite. Therefore, by Proposition \ref{discriminated}, for each existential sentence $\varphi$ satisfied by $G,$ there exists a group homomorphism $h \colon G \rightarrow \GL_n(\Z_p)$ such that $h(G)$ satisfies $\varphi.$ In particular, $\GL_n(\Z_p)$ satisfies the existential theory of $G.$
\end{remark}

The theorem follows from standard results in logic and it is well-known, but we include it for the sake of completeness. Indeed, models of the existential theory of $G$ are substructures of an ultrapower of $G$. More specifically the proof is based on Compactness Theorem and on Frayne's Theorem  (cf. \cite[Theorems 1.3.22 and 4.3.12]{CK}), although the latter's full strength is not used.

\begin{proof}[Proof of Theorem \ref{embeding}]
Let $\L$ be the first-order language of groups and let $\L_G$ be the expansion of $\L$ where we add a new constant for each element in $G,$ and let $(\mathfrak{G}, g)_{g \in G}$ be the model $G$ for $\L_G,$ with the obvious interpretation for the new constants. Let $\mathrm{Diag}(G)$ be the diagram of $(\mathfrak{G}, g)_{g \in G}$ (the collection of atomic sentences and negations of atomic sentences with parameters in $G$ that satisfies $G$), and let $I$ be the collection of finite subsets of $\textrm{Diag}(G).$ If $i = \{ \varphi_1, \dots, \varphi_n\} \in I,$ then $ \varphi = \varphi_1 \wedge \dots \wedge \varphi_n$ is a quantifier-free formula satisfied by $(\mathfrak{G}, g)_{g \in G}$ so, by Remark \ref{existentially closed}, there exists a group homomorphism $h_i \colon G \rightarrow H_i= \GL_n(\Z_p)$ such that $h_i(G)$ satisfies $\varphi.$  For each $i \in I$ set $A_i = \{ j \in I \mid j \supseteq i \},$ then $\{A_i \mid i \in I\}$ generates a proper filter, so we take a non-principal ultrafilter $\U$ containing it. Finally, let 
$$H := \sfrac{\prod_{i \in I} H_i}{\U} = \sfrac{\prod_{i \in I} \GL_n(\Z_p)}{\U}  = \GL_n(\Z_p)^\U,$$
and the injection $m \colon G \hookrightarrow H$ such that 
$$m(g) = \sfrac{(h_i(g))_{i \in I}}{\U} \in H,$$
which is a group monomorphism by \strokeL o\'s' Theorem (cf. \cite[Theorem 4.1.9]{CK}).
\end{proof}

We can deduce the main theorem from the preceding results. 

\begin{proof}[Proof of Theorem \ref{main}]
We can assume, without loss of generality, that $G$ is $R$-standard. According to Theorem \ref{embeding}, $G$ is a subgroup of the ultrapower $\GL_n(\Z_p)^\U.$ Moreover, 
$$\GL_{n}(\Z_p)^\U \cong \GL_{n}\left(\Z_p^\U \right) \leq \GL_{n}\left(\Q_p^\U \right)$$
(see \cite[1.L.6]{KeWe}). Thus, since $\Q_p^\U$ is a field, $G$ is linear.
\end{proof}


\end{document}